\documentclass[12pt,a4paper,leqno]{amsart}

\usepackage[applemac]{inputenc}
\usepackage[T1]{fontenc} 
\usepackage{amsthm}
\usepackage{amsfonts}       
\usepackage{amsmath}
\usepackage{amssymb}
\usepackage{mathrsfs}
\usepackage{xcolor}

\newcommand{\T}{\mathcal{T}}
\newcommand{\scrT}{\mathscr{T}}
\newcommand{\R}{\mathbb{R}}
\newcommand{\C}{\mathbb{C}}

\newcommand{\Z}{\mathbb{Z}}
\newcommand{\ud}{\mathrm{d}}

\newcommand{\F}{\mathscr{F}}

\newcommand{\D}{\mathscr{D}}

\newcommand{\si}{\sigma}

\newcommand{\G}{\mathscr{G}}

\newcommand{\nof}{{L^p(\sigma ; \ \! \ell^2)}}
\newcommand{\nog}{{L^{p'}(\omega)}}
\newcommand{\ch}{\text{ch}}
\newcommand{\w}{\widehat}

\theoremstyle{plain}
\newtheorem{theorem}{Theorem}
\newtheorem{prop}[theorem]{Proposition}
\newtheorem{lem}[theorem]{Lemma}

\newtheorem{con}[theorem]{Conjecture}
\newtheorem{q}[theorem]{Question}

\theoremstyle{definition}

\theoremstyle{remark}

\numberwithin{equation}{section}
\numberwithin{theorem}{section}

\author{Tuomas Hyt\"onen \and Emil Vuorinen}

\title{A Two-weight inequality between $L^p(\ell^2)$ and $L^p$}

\address{Department of Mathematics and Statistics, P.O.B. 68 (Gustaf H\"allstr\"omin katu 2b), FI-00014 University of Helsinki, Finland}
\email{tuomas.hytonen@helsinki.fi}
\email{emil.vuorinen@helsinki.fi}

\thanks{Both authors were supported by the ERC Starting Grant ``Analytic--probabilistic methods for borderline singular integrals''. They are members of the Finnish Centre of Excellence in Analysis and Dynamics Research.}

\begin{document}

\begin{abstract}
We consider boundedness of a certain positive dyadic operator 
$$
T^\si \colon L^p(\si; \ \! \ell^2) \to L^p(\omega),
$$ that arose during our attempts to develop a two-weight theory for the Hilbert transform in $L^p$. Boundedness of $T^\si$ is characterized when $p \in [2, \infty)$ in terms of certain testing conditions. This requires a new Carleson-type embedding theorem that is also proved. 
\end{abstract}

\maketitle

\section{Introduction}

This paper is an outgrowth of our attempts, so far incomplete, to develop a real-variable $L^p$-theory for two-weight inequalities of the Hilbert transform, which thus far has been achieved in the case $p=2$, by Lacey, Sawyer, Shen and Uriarte-Tuero \cite{Lacey:2wHilbert,LSSU:2wHilbert} (see also \cite{Hytonen:Hilbert}). The search for an $L^p$-analogue of certain intermediate results in the existing approach (op. cit.) to the $L^2$-theory led us to the present problem which, in our opinion, is natural and interesting in its own right. 

The problem we have in mind is that of characterising the boundedness of a certain positive bilinear form, which is in the spirit of the one appearing in the famous bilinear embedding theorem of Nazarov, Treil and Volberg \cite{NTV:bilinear} and its extension (from $L^2$ to $L^p$) by Lacey, Sawyer and Uriarte-Tuero \cite{LSU:positive}; these, in turn, are dyadic versions of an old theorem of Sawyer \cite{Sawyer:2wFractional}. The new feature that distinguishes our problem from those just mentioned is that we want to understand the boundedness not just on $L^p$ but on $L^p(\ell^2)$, the space of $L^p$ functions with values in $\ell^2$ or, if the reader prefers, a mixed-norm $L^p$ space. Recall that such spaces or norms frequently arise in the context of Littlewood--Paley theory, and this is also the prospective link of the new bilinear embedding theorem to the sought-after $L^p$-theory of the Hilbert transform.

While this link is pure speculation for the time being, our mixed-norm embedding seems independently interesting, both on the level of the result (a Sawyer-type testing, or ``local $T(1)$'', characterisation), and of the proof. The latter is a non-trivial modification of the successful parallel stopping cubes technology, adapted to the mixed-norm situation; among other things, this extension calls for a new Carleson embedding theorem, proved in Section \ref{embedding}, which might also have an independent interest. 

In order to give a more detailed discussion, we first need to set up some notation.

Fix a dimension $n$ of $\R^n$. Let $\sigma$ and $\omega$ be two locally finite non-negative Borel measures in $\R^n$. For every real number $a \in \R$ let $\delta_a$ denote the Dirac point mass at the point $a$.
Using the point masses, we define a measure on $(0,\infty)$ by $\eta:= \sum_{k \in \Z} \delta_{2^k}$. We equip $\R^{n+1}_+:= \R^n \times (0,\infty)$ with the product measure $\sigma \times \eta$.

A $\sigma \times \eta$-measurable function $f \colon \R^{n+1}_+ \to \C$  can be identified with the sequence $\{f_k\}_{k \in \Z}$ of Borel functions defined by $f_k(x) := f(x,2^{-k})$. Conversely,  a sequence $\{f_k\}_{k \in \Z}$ of Borel functions on $\R^n$ can be identified with the $\sigma \times \eta$--measurable function
$$
f(x,t):= \sum_{k \in \Z} 1_{\{2^{-k}\}}(t) f_k(x).
$$
For a set $A \subset \R^{n+1}_+, A \subset \R^n$ or $A \subset \R$, we write $1_A$ for its characteristic function.

Let $p \in [1, \infty)$. For a  $\sigma \times \eta$-measurable function $f$, we write 
\begin{equation}\label{def. of norm}
\| f \|_{L^p(\sigma ; \ \! \ell^2)}:= \bigg( \int_{\R^n} \Big( \sum_{k \in \Z} |f_k(x)|^2 \Big)^{\frac{p}{2}} \ud \sigma (x) \bigg)^\frac{1}{p}, 
\end{equation}
and the space $L^p(\sigma ; \  \! \ell^2)$ is defined to be the set of those $f$ such that \eqref{def. of norm} is finite. If 
$f$ is a 
$\sigma \times \eta$-measurable function, we write 
$$
| f |_{\ell^2} (x) := \Big( \sum_{k \in \Z} |f_k(x)|^2 \Big)^\frac{1}{2}, \quad x \in \R^n.
$$ 

For any Borel function $g$ on $\R^n$, we define
\begin{equation}\label{usual Lp}
\| g \|_{L^p(\omega)}:= \Big( \int_{\R^n} |g(x)|^p \ud \omega \Big)^{\frac{1}{p}}.
\end{equation}
The space $L^p(\omega)$ is the set of those $g$ such that \eqref{usual Lp} is finite. Using the measure $\si$ we define similarly $\| g \|_{L^p(\sigma)}$ and the space $L^p(\sigma)$.

Let $\D$ be the dyadic lattice 
$$
\D:= \big\{2^{-k}\big([0,1)^n+m\big) \colon k \in \Z, m \in \Z^n\big\}.
$$
For every $Q \in \D$ denote by $\widehat{Q}$ the Carleson box $Q \times (0, \ell(Q)]$, where $\ell(Q)$ is the side length of the cube $Q$.  Let $\mu$ be a fixed non-negative $\sigma \times \eta$-measurable function, and suppose that for each dyadic cube $Q \in \D$ there is associated a non-negative real number $\lambda_Q$. 

If $f \colon \R^{n+1}_+ \to [0, \infty)$ is  $\sigma \times \eta$-measurable and $g \colon \R^n \to [0, \infty)$ is a Borel function, we define
\begin{equation}\label{def. of form}
\Lambda(f, g ):= \sum_{Q \in \D} \lambda_{Q} \iint_{\widehat{Q}} f \mu \ \ud \eta \ud \si \int _{Q} g \ud \omega,
\end{equation}
and also for every $Q_0 \in \D$ the localized version
\begin{equation*}
\Lambda_{Q_0}(f, g ):= \sum_{\begin{substack}{Q \in \D \\ Q \subset Q_0}\end{substack}} \lambda_{Q} \iint_{\widehat{Q}} f \mu \ \ud \eta \ud \si \int _{Q} g \ud \omega.
\end{equation*}
The problem we are considering is when there exists a constant $C$ such that the inequality
\begin{equation}\label{question}
\Lambda (f,g) \leq C \| f \|_{L^p(\sigma ; \ \! \ell^2)} \| g \|_{L^{p'}(\omega)}
\end{equation}
 holds for all non-negative $f$ and $g$, where $p \in (1, \infty)$ and $p'$ is the H\"older conjugate of $p$. 
We emphasize that we consider the function $\mu$ and the coefficients $\{\lambda_Q\}_{Q \in \D}$ related to the definition of $\Lambda$ as fixed here.
 If such a constant $C$ exists then we may define $\Lambda(f,g)$ for every $f \in L^p(\sigma ; \ \! \ell^2)$ and $g \in L^{p'}(\omega)$ by \eqref{def. of form}, and \eqref{question} continues to hold for  these functions. Note that we could rephrase this problem equivalently by asking whether the operator
 $$
 T^\si f:= \sum_{Q \in \D} \lambda_{Q} \iint_{\widehat{Q}} f \mu \ \ud \eta \ud \si 1_Q
 $$
 is bounded from $L^p(\sigma ; \ \! \ell^2)$ into $L^p(\omega)$.
 
 If \eqref{question} holds, then $\| 1_{\w Q} \mu \|_{L^{p'}(\sigma ; \ \! \ell^2)}< \infty$ for every $Q \in \D$ such that $\lambda_Q, \sigma(Q)$ and $\omega(Q)$ are non-zero. Therefore, without changing the problem, we may assume that 
 \begin{equation}\label{assumption about mu}
 \|1_{\w Q} \mu \|_{L^{p'}(\sigma ; \ \! \ell^2)}< \infty, \quad \text{for every } Q \in \D. 
 \end{equation}
 
We answer this question when $p\geq2$ in terms of a testing characterization, i.e., we show that to have the inequality \eqref{question} it is enough to test it with a certain class of test functions. To get a precise meaning for this we next state our main theorem:

 \begin{theorem}\label{thm:main}
Let $p\in[2,\infty)$. For every $Q \in \D$ define the function
\begin{equation}\label{def. of testing functions}
\varphi_{Q}:= |1_{\widehat{Q}} \mu |_{\ell^{2}}^{p'-2}1_{\widehat{Q}} \mu,
\end{equation}
that satisfies  $\| \varphi_Q\|_\nof < \infty$ by \eqref{assumption about mu}.

Let $\T$ and $\T^*$ denote the smallest possible constants, with the understanding that they may be $\infty$, such that
\begin{equation}\label{testing1}
\Lambda_Q( \varphi_Q,g) \leq \T \| \varphi_Q\|_\nof \| g \| _\nog
\end{equation}
and 
\begin{equation}\label{testing2}
\Lambda_Q (f, 1_Q) \leq \T^* \| f \|_\nof \| 1_Q \|_\nog 
\end{equation}
hold for every $Q \in \D$ and every non-negative $\sigma \times \eta$-measurable function $f$ and non-negative Borel function $g$. Then there exist a constant $C<\infty$ such that \eqref{question} holds if and only if $\mathcal{T}+ \mathcal{T}^* < \infty$. Moreover, if  $\mathcal{T}+ \mathcal{T}^* < \infty$,  the smallest possible constant $\| \Lambda \|$ in \eqref{question} satisfies
$$
\| \Lambda \| \simeq \T + \T^*.
$$
 \end{theorem}
 
To prove Theorem \ref{thm:main} we use, as already mentioned,  the method of parallel stopping cubes. This technique was first introduced by Lacey, Sawyer, Shen and Uriarte-Tuero \cite{LSSU:2wHilbert} in an earlier arXiv version of their work, but replaced by other tools in the published paper. In \cite{Hytonen:A2survey} the parallel stopping cubes were used to study a similar problem but with usual $L^p$ norms rather than mixed ones. Our approach was to follow the outline of the proof in \cite{Hytonen:A2survey}, but in the set-up of this paper, it is not clear in the beginning what should be the  class of test functions in \eqref{testing1}. However, if one assumes that there exists a family $\{\varphi_Q\}_{Q \in \D}$ of test functions on $\R^{n+1}_+$ and starts to follow the outline of \cite{Hytonen:A2survey}, then there comes a situation that allows to guess the test functions, which leads to the definition \eqref{def. of testing functions}. Then it turns out, that these test functions are of the right form to conclude the proof. We show in the end of Section \ref{proof of main thm} how one can arrive at the definition \eqref{def. of testing functions}.

The case $p=2$ in Theorem \ref{thm:main} reduces to easier techniques. In fact, it can be seen as a special case of the result in \cite{NTV:bilinear}. The case $p \in (1,2)$ is an open problem, that we discuss more in Section \ref{discussion}, where we also state our conjecture about the two-weight inequality of the Hilbert transform in $L^p$.

For two numbers $\alpha, \beta \geq 0$ we use the notation $\alpha \lesssim \beta$ to mean that there exists an absolute constant $C$ such that $\alpha \leq C \beta$. Sometimes we write for example $\alpha \lesssim_p \beta$ to indicate that the implicit constant depends on $p$. Two sided estimates $\alpha \lesssim \beta \lesssim \alpha$ are abbreviated as $\alpha \simeq \beta$.

\section{An embedding theorem}\label{embedding}
In this section we start collecting tools to prove the main theorem \ref{thm:main}. In particular, we prove a Carleson-type embedding theorem that arises naturally during the proof in the next section.

We begin with a lemma that is the reason why we need to have $p \geq 2$ in Theorem \ref{thm:main}.

\begin{lem}\label{disjoint}
Let $p\geq 2$. Suppose $\{E_i\}_{i \in \mathcal{I}}$ is a countable collection of $\sigma \times \eta$-measurable sets such that $E_i \cap E_j = \emptyset$ if $i \not =j$. Let $f$ be a non-negative $\sigma \times \eta$-measurable function. Then
$$
\sum_{i \in \mathcal{I}} \| 1_{E_i} f \|_\nof ^p \leq \| f \|_\nof ^p.
$$
\end{lem}
 
\begin{proof}
Since $\frac{p}{2} \geq 1$, we have
\begin{equation*}
\begin{split}
\sum_{i \in \mathcal{I}} \| 1_{E_i} f \|_\nof ^p 
&= \int_{\R^n} \sum_{i \in \mathcal{I}}  \Big(\sum_{k \in \Z} 1_{E_i}(x,2^k) f(x,2^k)^2 \Big)^\frac{p}{2} \ud \sigma(x) \\
& \leq \int_{\R^n} \Big( \sum_{i \in \mathcal{I}} \sum_{k \in \Z} 1_{E_i}(x,2^k) f(x,2^k)^2 \Big)^\frac{p}{2} \ud \sigma(x)  \\
& \leq \| f \|_\nof^p.
\end{split}
\end{equation*}
\end{proof}

Next we state the well known dyadic Carleson embedding theorem that will be applied later. Let $\nu$ be a locally finite non-negative Borel measure in $\R^n$ and suppose $\{a_Q\}_{Q \in \D}$ is a collection of non-negative real numbers. We write the average over $Q \in \D$ of a Borel function $h \colon \R^n \to [0,\infty)$ as $ \langle h \rangle_Q^\nu:=\nu(Q)^{-1} \int_Q h \ud \nu$, that is understood to be zero if $\nu(Q)=0$. Let $p \in (1, \infty)$. There exists a constant $C$ such that
\begin{equation}\label{dyadic Carleson}
\sum_{Q \in \D} \big( \langle h \rangle^\nu_Q\big)^p a_Q \leq C\int_{\R^n}  h ^p \ud \nu
\end{equation}
holds for all Borel functions $h \colon \R^n \to [0, \infty)$ if and only if there exists a constant $C'$ such that
\begin{equation}\label{Carleson condition}
\sum_{\begin{substack}{Q' \in \D \\ Q' \subset Q}\end{substack}} a_{Q'} \leq C' \nu(Q)
\end{equation}
holds for all $Q \in \D$. Moreover, the smallest possible constants in \eqref{dyadic Carleson} and \eqref{Carleson condition} satisfy $C \simeq_p C'$.

\subsection*{Stopping cubes}
Here we show how to construct the collections of stopping cubes relevant to the present purposes.
Let $Q_0 \in \D$ and let $g \colon \R^n \to [0,\infty)$ be a locally $\omega$-integrable function. Set $\G_0:= \{ Q_0\}$, and suppose that the collections $\G_j, j \in \{0,1, \dots, k\}$, are defined for some $k$. If $G \in \G_k$, we define $\ch_\G(G)$ to be the collection of maximal dyadic cubes $Q \in \D$ such that $Q \subset G$ and
$
\langle g \rangle_Q^\omega > 2 \langle g \rangle_G^\omega.
$
Then we set $\G_{k+1}:= \bigcup _{G \in \G_k} \ch_\G(G)$, and the collection of stopping cubes with the top cube $Q_0$ is defined as $\G:= \bigcup_{k=0}^\infty \G_k$.

If $Q \in \D, Q \subset Q_0,$ we denote by $\pi_\G (Q)$ the smallest cube $G \in \G$ that contains $Q$. From the definition of $\G$ it is seen that 
$$
\langle g \rangle_Q ^\omega \leq 2 \langle g \rangle_{\pi_\G (Q)}^\omega.
$$ 
It follows from the construction that $\G$ is a \emph{2-Carleson family} (with respect to $\omega$), which means that for every $G \in \G$  there holds
$$
\sum_{\begin{substack}{G' \in \G \\ G' \subset G}\end{substack}} \omega(G') \leq 2 \omega(G).
$$
This combined with the  dyadic Carleson embedding theorem stated above implies that
\begin{equation}\label{Carleson for g}
\sum_{G \in \G} \big(\langle h \rangle^\omega_G\big)^p \omega(G) \lesssim_p \int h^p \ud \omega
\end{equation}
holds for every Borel function $h \colon \R^n \to [0, \infty)$ and every $p \in (1, \infty)$.

Let then $f \colon \R^{n+1}_+ \to [0,\infty)$ be a $\sigma \times \eta$-measurable function such that 
$$
\iint_{\widehat{Q_0}} f \mu \ud \eta \ud \si<\infty,
$$
where again $Q_0 \in \D$ is some fixed cube. We want to define a similar collection of cubes for the function $f$ involving the test functions $\varphi_Q$ from \eqref{def. of testing functions}. The reason why we define the collection as follows becomes more apparent when one studies what happens in   the equations \eqref{substitute test function} and \eqref{apply testing1} below in the proof of the main theorem. First set $\F_0:= \{Q_0\}$, and suppose $\F_0, \dots, \F_k$ are defined for some $k$. Let $F \in \F_k$. We define $\ch_\F(F)$ to be the set of maximal cubes $Q \in \D$ such that $Q \subset F$ and
\begin{equation}\label{stopping for f}
 \frac{\iint_{\widehat{Q}} f \mu \ud \eta \ud \si}{\iint_{\widehat{Q}} \varphi_{F} \mu \ud \eta \ud \si}
>A \frac{\iint_{\widehat{F}} f \mu \ud \eta \ud \si}{\iint_{\widehat{F}} \varphi_{F} \mu \ud \eta \ud \si},
\end{equation}
where $A>0$ is a big enough constant to be specified during the proof of the main theorem in Section \ref{proof of main thm}. Then $\F_{k+1}:= \bigcup_{F \in \F_k} \ch_\F (F)$ and the collection of stopping cubes for $f$ with the top cube $Q_0$  is $\F :=\bigcup_{k=0}^\infty \F_k$.

If $Q \in \D, Q \subset Q_0$, we denote by $\pi_\F(Q)$ the smallest cube $F \in \F$ that contains $Q$. It follows from the construction of $\F$ that
\begin{equation}\label{stopping property}
\frac{\iint_{\widehat{Q}} f \mu \ud \eta \ud \si}{\iint_{\widehat{Q}} \varphi_{F} \mu \ud \eta \ud \si}
\leq A \frac{\iint_{\widehat{F}} f \mu \ud \eta \ud \si}{\iint_{\widehat{F}} \varphi_{F} \mu \ud \eta \ud \si},\qquad F=\pi_\F(Q).
\end{equation}
Related to these define for $Q \in \D$ the average-type quantity
\begin{equation}\label{"average"}
[ f ]_Q:= \frac{\iint_{\widehat{Q}} f \mu \ud \eta \ud \si}{\iint_{\widehat{Q}} \varphi_{Q} \mu \ud \eta \ud \si}.
\end{equation}

\subsection*{}

For later use we record here a few identities related to the test functions $\varphi_Q$. Namely, a direct computation shows that 
\begin{equation}\label{eq:DirectComputation}
\iint_{\widehat{Q}} \varphi_{Q} \mu \ud \eta \ud \si = \int_Q | 1_{\w Q} \mu|_{\ell^2}^{p'} \ud \si= \| 1_{\w Q } \mu \|_{L^{p'}(\sigma ; \ \! \ell^2)}^{p'} =\| \varphi_Q \|_\nof^p.
\end{equation}

Now we are ready for the embedding theorem that is the main result of this section.

\begin{prop}\label{embedding theorem}
Let $p\in[2,\infty)$ and $Q_0 \in \D$. Let $f \colon \R^{n+1}_+  \to [0,\infty)$ be a $\sigma \times \eta$-measurable function such that
$$
\iint_{\widehat{Q_0}} f \mu \ud \eta \ud \si<\infty.
$$
Let $\F$ be the collection stopping cubes for the function $f$ with the top cube $Q_0$ as described above. Then
\begin{equation}\label{Carleson for f}
\sum_{F \in \F} [f ]_F^p \| \varphi_F \|_{L^p(\sigma ; \ \! \ell^2)}^p  \lesssim_p \| f \|_{L^p(\sigma ; \ \! \ell^2)}^p.
\end{equation}
\end{prop}

It is important to note that here we need the fact $p\geq 2$, because in the proof we apply Lemma \ref{disjoint} that does not hold for $p \in (1,2)$.

Inequality \eqref{Carleson for f} somewhat resembles Inequality \eqref{dyadic Carleson}, and we shall actually interpret the left hand side of \eqref{Carleson for f} in a way that allows us to apply the dyadic Carleson embedding theorem.	

\begin{proof}[Proof of Proposition \ref{embedding theorem}]
Our first goal is to show that for every $F \in \F$ there holds
\begin{equation}\label{sparse}
\sum_{\begin{substack}{F' \in \ch _\F (F)}\end{substack}}\| \varphi_{F'}\|_\nof^p 
\leq \frac{1}{2} \| \varphi_{F} \|^p_\nof
\end{equation}
if the parameter $A$ related to the construction of $\F$ is  big enough. Recall the identity $\| \varphi_Q\|_\nof^p=\int_Q |1_{\w Q} \mu |_{\ell^2}^{p'} \ud \sigma.$ Fix a cube $F \in \F$ and define
$$
H:= \big\{ x \in F \colon |1_{\w F} \mu |_{\ell^2}(x) > B \sum_{F' \in \ch_\F(F)} |1_{\w {F'}} \mu |_{\ell^2}(x) \big\},
$$
where $B>0$ is a big constant that will be fixed soon. Then there holds
\begin{equation}\label{outside H}
\begin{split}
\sum_{F' \in \ch_\F(F)}\int_{F'\cap H} |1_{\w {F'}} \mu |_{\ell^2}^{p'} \ud \sigma 
&\leq \sum_{F' \in \ch_\F(F)} B^{-p'} \int_{F'\cap H} |1_{\w F} \mu |_{\ell^2}^{p'} \ud \sigma \\
& \leq B^{-p'} \int_{F} |1_{\w F} \mu |_{\ell^2}^{p'} \ud \sigma,
\end{split}
\end{equation}
since the cubes $F' \in \ch_\F (F)$ are pairwise disjoint.

On the other hand, because $p'-2\leq 0$,  we can estimate in $F' \setminus H$, where $F' \in \ch_\F(F)$, that
\begin{equation*}
\begin{split}
\int_{F' \setminus H} |1_{\w {F'}} \mu |_{\ell^2}^{p'} \ud \sigma 
=\int_{F' \setminus H} |1_{\w {F'}} \mu |_{\ell^2}^{p'-2} |1_{\w {F'}} \mu |_{\ell^2}^2 \ud \sigma
&\leq B^{2-p'}\int_{F' \setminus H} |1_{\w F} \mu |_{\ell^2}^{p'-2} |1_{\w {F'}} \mu |_{\ell^2}^2 \ud \sigma \\
&\leq B^{2-p'} \iint_{\widehat{F'}} \varphi_{F} \mu \ud \eta \ud \si.
\end{split}
\end{equation*}
Hence, the stopping condition \eqref{stopping for f} gives
\begin{equation}\label{in H}
\begin{split}
\sum_{F' \in \ch_\F(F)}\int_{F' \setminus H} |1_{\w {F'}} \mu |_{\ell^2}^{p'} \ud \sigma 
&\leq B^{2-p'}\sum_{F' \in \ch_\F(F)} \iint_{\widehat{F'}} \varphi_{F} \mu \ud \eta \ud \si \\
& \leq B^{2-p'}A^{-1} \frac{\iint_{\widehat{F}} \varphi_{F} \mu \ud \eta \ud \si}{\iint_{\widehat{F}} f \mu \ud \eta \ud \si} \sum_{F' \in \ch_\F(F)}\iint_{\widehat{F'}} f \mu \ud \eta \ud \si \\
&\leq B^{2-p'}A^{-1} \iint_{\widehat{F}} \varphi_{F} \mu \ud \eta \ud \si \\
&= B^{2-p'}A^{-1}\int_{F} |1_{\w {F}} \mu |_{\ell^2}^{p'} \ud \sigma. 
\end{split}
\end{equation}

Combining estimates \eqref{outside H} and \eqref{in H} with the identity $\| \varphi_F\|_\nof^p=\int_F |1_{\w F} \mu |^{p'} \ud \sigma$ we have
$$
\sum_{F' \in \ch_\F(F)}\| \varphi_{F'}\|_\nof^p \leq \big( B^{-p'}+B^{2-p'}A^{-1}) \| \varphi_{F}\|_\nof^p.
$$
From here it is seen that if we choose for example $B :=4^{\frac{1}{p'}}$ and $A:= 4B^{2-p'}$, then $\eqref{sparse}$ is satisfied. By summing a geometric series, from \eqref{sparse} it follows that
\begin{equation}\label{geometric}
\sum_{\begin{substack}{F' \in \F \\ F' \subset F}\end{substack}}\| \varphi_{F'}\|_\nof^p 
\leq 2 \| \varphi_{F} \|^p_\nof
\end{equation}
holds for every $F \in \F$.

Next we view the sum $\sum_{F \in \F} [f ]_F^p \| \varphi_F \|_{L^p(\sigma ; \ \! \ell^2)}^p$ in a way that allows to apply the dyadic Carleson embedding theorem. If $F \in \F$, we write 
$$E_\F(\w F):= \w F \setminus \bigcup_{F' \in \ch_\F (F)} \w {F'}.$$
Note that 
 the sets $E_\F(\w F)$ are pairwise disjoint. Moreover, there holds 
$$
\w F= \bigcup_{\begin{substack}{F' \in \F \\ F' \subset F} \end{substack}} E_\F(\w F')
$$
for every $F \in \F$. 

Define a measure $\nu$ on $\R^{n+1}$ by
$$
  \nu:= \sum_{F \in \F} \| \varphi_F\|_\nof^p \delta_{z(F)},\qquad z(F):=\big(\operatorname{centre}(F),\frac34\ell(F)\big), 
$$
where $\operatorname{centre}(F)$ is the centre of the $n$-dimensional cube $F$, and $z(F)$ is the centre of the upper-half of the $(n+1)$-dimensional cube $\widehat F$.
Define also a function $\alpha$ on $\R^{n+1}$ by 
$$
\alpha:= \sum_{F \in \F} \frac{\iint_{E_\F(\widehat{F})} f \mu \ud \eta \ud \si}{\| \varphi_F\|_\nof^p}1_{\{z(F)\}}, 
$$
and recall that $\iint_{\widehat{F}} \varphi_{F} \mu \ud \eta \ud \si =\| \varphi_F \|_\nof^p$ by Equation \eqref{eq:DirectComputation}.

Let $F \in \F$. Then, by \eqref{geometric}, there holds that
$$
\nu(\widehat F) = \sum_{\begin{substack}{F' \in \F \\ F' \subset F}\end{substack}}\| \varphi_{F'}\|_\nof^p \simeq \| \varphi_{F}\|_\nof^p,
$$
and thus also
\begin{equation}\label{nu Carleson}
\sum_{\begin{substack}{F' \in \F \\ F' \subset F}\end{substack}} \nu (\widehat{F'}) 
\simeq \sum_{\begin{substack}{F' \in \F \\ F' \subset F}\end{substack}} \| \varphi_{F'}\|_\nof^p 
=\nu(\widehat F).
\end{equation}
This says that the collection $\widehat{\F}:=\{\widehat F:F\in\F\}$ is a Carleson family with respect to the measure $\nu$.

Notice that for every $F \in \F$ we have
$$
\iint_{\widehat{F}} f \mu \ud \eta \ud \si
= \sum_{\begin{substack}{F' \in \F \\F' \subset F}\end{substack}} \frac{\iint_{E_\F(\widehat{F'})} f \mu \ud \eta \ud \si}{ \| \varphi_{F'}\|_\nof^p} \| \varphi_{F'}\|_\nof^p
= \int_{\widehat F} \alpha \ud \nu,
$$
and hence
$$
[ f ]_F ^p \| \varphi_F \|_\nof^p=\Big(\frac{\iint_{\widehat{F}} f \mu \ud \eta \ud \si}{\| \varphi_{F}\|_\nof^p} \Big)^p \| \varphi_{F}\|_\nof^p 
\simeq \Big( \frac{\int_{\widehat F} \alpha \ud \nu}{\nu({\widehat F})} \Big)^p \nu({\widehat F}).
$$

We can now apply the dyadic Carleson embedding theorem to conclude that
\begin{equation*}
\begin{split}
\sum_{F \in \F} [f ]_F^p \| \varphi_F \|_\nof^p 
\simeq \sum_{F \in \F} \Big( \frac{\int_{\widehat F} \alpha \ud \nu}{\nu({\widehat F})} \Big)^p \nu({\widehat F}) 
&\lesssim \int_{\R^{n+1}} \alpha^p \ud \nu.
\end{split}
\end{equation*}
Writing out the definition of $\alpha$ and $\nu$ we have
\begin{equation*}
\begin{split}
\int_{\R^{n+1}} \alpha^p \ud \nu 
&=\sum_{F \in \F} \Big(  \frac{\iint_{E_\F(\widehat{F})} f \mu \ud \eta \ud \si}{ \| \varphi_{F}\|_\nof^p}\Big)^p \| \varphi_{F}\|_\nof^p \\
& \leq \sum_{F \in \F} \Big(  \frac{\| 1_{E_\F(\w F)} f \|_\nof \|1_{\w F} \mu \|_{L^{p'}(\si ; \ \! \ell^2)}}{ \| \varphi_{F}\|_\nof^p}\Big)^p \| \varphi_{F}\|_\nof^p \\
& = \sum_{F \in \F} \| 1_{E_\F(\w F)} f \|_\nof^p \leq \| f \|_\nof ^p,
\end{split}
\end{equation*}
where we used the identity $ \| \varphi_F \|_\nof^p= \| 1_{\w F } \mu \|_{L^{p'}(\sigma ; \ \! \ell^2)}^{p'}$ and applied Lemma \ref{disjoint}. This concludes the proof.

\end{proof}

\section{Proof of Theorem \ref{thm:main}}\label{proof of main thm}

In this section we prove Theorem \ref{thm:main}. After the proof we show how one can arrive at the definition \eqref{def. of testing functions} of the test functions $\varphi_Q$.

\begin{proof}[Proof of Theorem \ref{thm:main}]
If \eqref{question} holds, then it is clear that the testing conditions \eqref{testing1} and \eqref{testing2} hold, and that $ \max (\T,\T^*) \leq \| \Lambda \|$. Hence  we can focus on the converse, that is, we assume that $\T, \T^*< \infty$ and show that $\| \Lambda \| \lesssim \T + \T^*$. By monotonicity it is enough to fix an arbitrary cube $Q_0 \in \D$ and two non-negative functions $f \in L^p(\si ; \ \! \ell^2)$ and $g \in L^{p'}(\omega)$, and to show that
\begin{equation}\label{est:main}
\sum_{\begin{substack}{Q \in \D \\ Q \subset Q_0} \end{substack}}\lambda_{Q} \iint_{\widehat{Q}} f \mu \ \ud \eta \ud \si \int _{Q} g \ud \omega 
\lesssim (\T+\T^*) \| f \|_\nof \| g \|_\nog.
\end{equation}

Let $\F$ and $\G$ be the collections of stopping cubes for the functions $f$ and $g$, respectively, with the top cube $Q_0$ as described in Section \ref{embedding}. Using $\F$ and $\G$ we can reorganize the sum in the left hand side of \eqref{est:main}. If $Q \in \F, Q \subset Q_0$, then there exists a unique pair $(F,G) \in \F \times \G$, denoted by $\pi (Q)$,  such that $\pi_\F(Q)=F$ and $\pi_\G(Q)=G$.  Since in this case clearly $F \cap G \not = \emptyset$, it follows from the properties of dyadic cubes that either $F \subset G$ or $G \subset F$. Hence it is seen that the left hand side of \eqref{est:main} satisfies
\begin{equation*}
 LHS\eqref{est:main} = \sum_{F \in \F}\sum_{\begin{substack}{G \in \G \\ G \subset F}\end{substack}} \sum_{\begin{substack}{Q \in \D \\ \pi (Q) = (F,G)}\end{substack}} 
 +\sum_{G \in \G}\sum_{\begin{substack}{F \in \F \\ F \subsetneq G}\end{substack}} \sum_{\begin{substack}{Q \in \D \\ \pi (Q) = (F,G)}\end{substack}}
 =: I + II. 
\end{equation*}
The proof divides into considering the parts $I$ and $II$ separately.

\subsection*{Estimate for $II$}  
For $G \in \G$ define the collection
$$
\ch ^*_\G (G):= \big\{ G' \in \ch_\G(G) \colon \pi_\F (G') \subset G \big\}. 
$$
Also,  write $E_\G(\w G):= \w G \setminus \bigcup_{G' \in \ch_\G (G)} \w {G'}$. 

Let $Q \in \D, F \in \F$ and $G \in \G$ be such that $F \subsetneq G$ and $\pi (Q)=(F,G)$.  Note first that
$$
\w G = E_\G(\w G) \cup \bigcup_{G' \in \ch_\G(G)} \w{G'}.
$$
Because $Q \subset G$, and accordingly $\w Q \subset \w G$, this implies that 
$$
\w Q = \big(E_\G(\w G) \cap \w Q\big) \cup \bigcup_{\begin{substack}{G' \in \ch_\G(G)}\end{substack}} \big(\w{G'} \cap \w Q\big).
$$
Let $G' \in \ch_\G (G)$.  If $G' \cap Q = \emptyset$, then clearly $\w {G'} \cap \w Q = \emptyset$. Assume $G' \cap Q \not = \emptyset$. Then, since $\pi_\G (Q)=G$, it must be that $G' \subsetneq Q$. Also, since $G' \subsetneq Q \subset F \subsetneq G$, we can conclude that $G' \in \ch_\G^*(G)$. This reasoning shows that actually
\begin{equation}\label{decompose Q}
\w Q = \big(E_\G(\w G) \cap \w Q\big) \cup \bigcup_{\begin{substack}{G' \in \ch^*_\G(G) \\ G'\subsetneq Q}\end{substack}} \w{G'},
\end{equation}
where one should note that the sets $E_\G(\w {G})$ and $\widehat{G'}$, $G' \in \ch_\G(G)$, are pairwise disjoint.

If $Q\in \D, F \in \F$ and $G \in \G$ are such that $F \subsetneq G$ and $\pi(Q)=(F,G)$, we can write in view of \eqref{decompose Q} that
$$
\iint_{\w Q} f \mu \ud \eta \ud \sigma= \iint_{\w Q} f_G \mu \ud \eta \ud \sigma,
$$
where
\begin{equation*}
f_G\colon = 1_{E_\G(\w G)} f
+ \sum_{\begin{substack}{G' \in \ch_\G^* (G)}\end{substack}} \frac{\iint_{\w {G'}} f \mu \ud \sigma}{\iint_{\w {G'}} \varphi_{\pi_\F (G')} \mu \ud \sigma} 1_{\w {G'}} \varphi_{\pi_\F (G')}.
\end{equation*}
Hence
\begin{equation}\label{apply dual testing}
\begin{split}
II&
\leq 2\sum_{G \in \G} \langle g \rangle^{\omega}_{G}\sum_{\begin{substack}{F \in \F \\ F \subsetneq G}\end{substack}} 
\sum_{ \begin{substack}{Q \in \D \\ \pi (Q) = (F, G)} \end{substack} } \lambda_Q \iint_{\widehat{Q}} f_G \mu \ \ud \eta \ud \si \int _{Q} 1_G \ud \omega \\ 
& \leq 2 \sum_{G \in \G} \langle g \rangle ^\omega_G \Lambda_G (f_G , 1_G)  
\leq 2 \T^* \sum_{G \in \G} \langle g \rangle^{\omega}_{G} \| f_{G} \|_{L^{p}(\si ; \ \! \ell^{2})} \omega(G)^{\frac{1}{p'}} \\
& \leq 2 \T^* \Big( \sum_{G \in \G} \| f_{G} \|_{L^{p}(\si ; \ \! \ell^{2})}^{p} \Big)^{\frac{1}{p}}\Big( \sum_{G \in \G} \big( \langle g \rangle^{\omega}_{G} \big)^{p'}\omega(G) \Big)^{\frac{1}{p'}}.
\end{split}
\end{equation}
Equation \eqref{Carleson for g} gives that 
\begin{equation}\label{apply Carleson for g}
\Big( \sum_{G \in \G} \big( \langle g \rangle^{\omega}_{G} \big)^{p'}\omega(G) \Big)^{\frac{1}{p'}} \lesssim \| g \|_{L^{p'}(\omega)}.
\end{equation}

Fix some $G \in \G$.  We have 
\begin{equation*}
\begin{split}
\| f_{G} \|_{L^{p}(\si ; \ \! \ell^{2})}
&\leq \| 1_{E_\G(\w G)}f \|_{L^{p}(\si ; \ \! \ell^{2})}
+\Big\|\sum_{\begin{substack}{G' \in \ch_\G^* (G)}\end{substack}} \frac{\iint_{\w {G'}} f \mu \ud \sigma}{\iint_{\w {G'}} \varphi_{\pi_\F (G')} \mu \ud \sigma} 1_{\w {G'}} \varphi_{\pi_\F (G')}\Big\|_{L^p(\sigma ; \ \! \ell^2)}.
\end{split}
\end{equation*}
For every $G' \in \ch ^*_\G (G)$ there exists a cube $F \in \F$ such that $G' \subset F \subset G$ and $\pi_\F (G')=F$, whence it follows that $\pi_\G(F)=G$ or $F \in \ch_\G(G)$. Hence the sum over $G' \in \ch^*_\G(G)$ can be written as 
\begin{equation}\label{organize}
\sum_{G' \in \ch_\G^*(G)}= \sum_{\begin{substack}{F \in \F \colon \\ \pi_\G (F)=G  \text{ or} \\ F \in \ch_\G(G) }\end{substack}}\sum_{\begin{substack}{G' \in \ch _\G(G) \\ \pi_\F (G') =F}\end{substack}}.
\end{equation}
Also, the stopping condition implies for $G' \in \ch_\G^*(G)$ that
\begin{equation}\label{stopping gives}
\frac{\iint_{\w {G'}} f \mu \ud \sigma}{\iint_{\w {G'}} \varphi_{\pi_\F (G')} \mu \ud \sigma}  \leq A[ f ]_{\pi_\F(G')}.
\end{equation}
Since the cubes $G' \in \ch_\G^*(G)$ are pairwise disjoint,  \eqref{organize} and \eqref{stopping gives} give that
\begin{equation*}
\begin{split}
\Big\|\sum_{\begin{substack}{G' \in \ch_\G^* (G) }\end{substack}} &\frac{\iint_{\w {G'}} f \mu \ud \sigma}{\iint_{\w {G'}} \varphi_{\pi_\F (G')} \mu \ud \sigma} 1_{\w {G'}} \varphi_{\pi_\F (G')}\Big\|_{L^p(\sigma ; \ \! \ell^2)}^p \\
& \leq A^p \sum_{\begin{substack}{F \in \F \colon \\ \pi_\G (F)=G  \text{ or} \\ F \in \ch_\G(G) }\end{substack}} [f]_F^p \sum_{\begin{substack}{G' \in \ch _\G(G) \\ \pi_\F (G') =F}\end{substack}} \| 1_{\w {G'}} \varphi_F \|^p_{L^p(\sigma ; \ \! \ell^2)} \\
& \leq A^p \sum_{\begin{substack}{F \in \F \colon \\ \pi_\G (F)=G  \text{ or} \\ F \in \ch_\G(G) }\end{substack}} [f]_F^p \| \varphi_F \|_{L^p(\sigma ; \ \! \ell^2)}^p.
\end{split}
\end{equation*}

Using the estimate for $\| f_G \|_\nof$ we get
\begin{equation*} 
\begin{split}
\Big( \sum_{G \in \G} \| f_{G} \|_{L^{p}(\si ; \ \! \ell^{2})}^{p} \Big)^{\frac{1}{p}} 
 &\leq \Big( \sum_{G \in \G} \| 1_{E_\G(\w G)}f \|_{L^{p}(\si ; \ \! \ell^{2})}^p \Big)^\frac{1}{p} \\
&+ A\Big( \sum_{G \in \G}  \sum_{\begin{substack}{F \in \F \colon \\ \pi_\G (F)=G  \text{ or} \\ F \in \ch_\G(G) }\end{substack}} [f]_F^p \| \varphi_F \|_{L^p(\sigma ; \ \! \ell^2)}^p\Big)^\frac{1}{p}. \\
\end{split}
\end{equation*}
Lemma \ref{disjoint} implies that  $\sum_{G \in \G} \| 1_{E_\G(\w G)}f \|_{L^{p}(\si ; \ \! \ell^{2})}^p \leq \| f \|_\nof^p$.  Since for every $F \in \F$ there exist at most two cubes $G \in \G$ such that $\pi_\G (F)=G$ or $F \in \ch_\G(G)$, Proposition \ref{embedding theorem} gives
$$
\sum_{G \in \G}  \sum_{\begin{substack}{F \in \F \colon \\ \pi_\G (F)=G  \text{ or} \\ F \in \ch_\G(G) }\end{substack}} [f]_F^p \| \varphi_F \|_{L^p(\sigma ; \ \! \ell^2)}^p
\leq 2\sum_{F \in \F}[f]_F^p \| \varphi_F \|_\nof^p
\lesssim \| f \|_\nof^p.
$$

Hence we have shown that $\sum_{G \in \G} \| f_{G} \|_{L^{p}(\si ; \ \! \ell^{2})}^{p} \lesssim \| f \|_\nof^p$, and combining this with \eqref{apply dual testing} and \eqref{apply Carleson for g} yields
$$
II \lesssim \T^* \| f \|_{L^p(\sigma ; \ \! \ell^2)} \| g \|_{L^{p'}(\omega)}.
$$

\subsection*{Estimate for $I$} Similarly as with the cubes $G \in \G$ we define for $F \in \F$ the collection
$$
\ch^*_\F(F):= \big\{ F' \in \ch_\F(F) \colon \pi_\G (F') \subset F\big\}.
$$
Denote $E_\F(F):= F \setminus \bigcup_{F' \in \ch_\F(F)} F'$. Suppose $Q \in \D, F \in \F$ and $G \in \G$ are such that $G \subset F$ and $\pi(Q)=(F,G)$. Then, by a similar reasoning as above when estimating the term $II$, there holds that
$$
\int_Q g \ud \omega = \int_Q g_F \ud \omega,
$$
where
\begin{equation*}
g_{F}:= 1_{E_\F(F)}g + \sum_{F' \in \text{ch}^{*}_{\F}(F)} \langle g \rangle^{\omega}_{F'} 1_{F'}.
\end{equation*}
Also, the construction of $\F$ shows that
\begin{equation}\label{substitute test function}
\begin{split}
\iint_{\widehat{Q}} f \mu \ud \eta \ud \si
&= \frac{\iint_{\widehat{Q}} f \mu \ud \eta \ud \si}{\iint_{\widehat{Q}} \varphi_{F} \mu \ud \eta \ud \si} \iint_{\widehat{Q}} \varphi_F  \mu \ud \eta \ud \si \\
&\leq A [ f ]_{F} \iint_{\widehat{Q}} \varphi_{F} \mu \ud \eta \ud \si. 
\end{split}
\end{equation}

Using these we have
\begin{equation}\label{apply testing1}
\begin{split}
I &\leq A \sum_{F \in \F} [ f ]_F \sum_{\begin{substack}{G \in \G \\ G \subset F}\end{substack}} \sum_{\begin{substack}{Q \in \D \\ \pi (Q) = (F,G)}\end{substack}} \lambda_Q \iint_{\w Q} \varphi_F \mu \ud \si
\int_Q g_F \ud \omega \\
& \leq A \sum_{F \in \F} [f]_F \Lambda_F(\varphi_F, g_F) 
\leq A \T \sum_{F \in \F} [f]_F \| \varphi_F \|_{L^p(\sigma ; \ \! \ell^2)} \| g_F \|_{L^{p'}(\omega)} \\
& \leq A \T  \Big( \sum_{F \in \F} [ f ]_F^p \| \varphi_F\|_{L^p(\sigma ; \ \! \ell^2)} ^p \Big)^\frac{1}{p} \Big( \sum_{F \in \F} \| g_F \|_{L^{p'}(\omega)}^{p'} \Big)^{\frac{1}{p'}}.
\end{split}
\end{equation}
Proposition \ref{embedding theorem} gives again that
$$
\Big( \sum_{F \in \F} [ f ]_F^p \| \varphi_F\|_{L^p(\sigma ; \ \! \ell^2)} ^p \Big)^\frac{1}{p}
\lesssim \| f \|_{L^p(\sigma ; \ \! \ell^2)}.
$$

To conclude the proof it remains to consider $\sum_{F \in \F} \| g_F \|_{L^{p'}(\omega)}^{p'}$. If $F \in \F$, then
$$
\| g_F \|_{L^{p'}(\omega)}^{p'} = \| 1_{E_\F(F)}g\|_{L^{p'}(\omega)}^{p'}+ \sum_{F' \in \ch_\F^*(F)} \big(\langle g \rangle_{F'}^\omega\big)^{p'} \omega (F').
$$
Clearly 
$$
\sum_{F \in \F} \| 1_{E_\F(F)}g \|_{L^{p'}(\omega)}^{p'}  \leq \| g \|_\nog^{p'},
$$
since the sets $E_\F(F), F \in \F$, are pairwise disjoint.  Rewriting the sum as in \eqref{organize}, the other term satisfies
\begin{equation*}
\begin{split}
\sum_{F \in \F} \sum_{F' \in \ch_\F^*(F)} \big(\langle g \rangle_{F'}^\omega\big)^{p'} \omega (F') 
&\leq  2^{p'}\sum_{F \in \F} \sum_{\begin{substack}{G \in \G \colon \\ \pi_\F (G)=F  \text{ or} \\ G \in \ch_\F(F) }\end{substack}} \big( \langle g \rangle_{G}^\omega\big)^{p'}\sum_{\begin{substack}{F' \in \ch_\F(F) \\ \pi_\G (F')=G}\end{substack}} \omega (F') \\
&\leq 2^{p'} \sum_{F \in \F} \sum_{\begin{substack}{G \in \G \colon \\ \pi_\F (G)=F  \text{ or} \\ G \in \ch_\F(F) }\end{substack}} \big( \langle g \rangle_{G}^\omega\big)^{p'} \omega (G) \\
& \leq 2^{1+p'} \sum_{G \in \G} \big( \langle g \rangle_{G}^\omega\big)^{p'} \omega (G) \lesssim \| g \|_\nog^{p'}.
\end{split}
\end{equation*}
Thus we have shown that $\sum_{F \in \F} \| g_F \|_{L^{p'}(\omega)}^{p'} \lesssim \| g \|_\nog^{p'},$ and this concludes the proof of Theorem \ref{thm:main}.

\end{proof}

Let us now discuss how one can arrive at the definition \eqref{def. of testing functions} of the test functions $\varphi_Q$. Suppose we want to find a family $\{\theta_Q\}_{Q \in \D}$ of non-negative $\sigma \times \eta$-measurable functions such that if 
\begin{equation}\label{theta testing1}
\Lambda_Q (\theta_Q,g) \leq C_1 \| \theta_Q\|_\nof \| g \|_{L^{p'}(\omega)}
\end{equation}
and
\begin{equation}\label{theta testing 2}
\Lambda_Q(f, 1_Q) \leq C_2 \| f \|_\nof \| 1_Q \|_\nog
\end{equation}
hold uniformly for $Q \in \D$ and non-negative functions $f$ and $g$, then
$$
\Lambda(f,g) \lesssim (C_1+C_2) \| f \|_\nof \| g \|_\nog
$$
holds for all non-negative $f$ and $g$.  To find this kind of family, we first assume that $\{\theta_Q\}_{Q \in \D}$ is some collection of functions such that  \eqref{theta testing1} and \eqref{theta testing 2} hold, and then try to follow the method of parallel stopping cubes. We can proceed precisely as in Section \ref{proof of main thm} until we have to prove the estimate
\begin{equation}\label{embedding with theta}
\sum_{F \in \F} \Big(\frac{\iint_{\widehat{F}} f \mu \ud \eta \ud \si}{\iint_{\widehat{F}} \theta_{F} \mu \ud \eta \ud \si}\Big)^p \| \theta_F\|_\nof^p
\lesssim \| f \|_{L^p(\sigma ; \ \! \ell^2)}^p,
\end{equation}
where now the collection $\F$ is defined with the functions $\theta_Q$ instead of $\varphi_Q$.

To prove \eqref{embedding with theta}, we might want to minimize the ratios
$$
\frac{\| \theta_{F} \|_{L^{p}(\si ; \ \! \ell^{2})}^{p}}{\left(\iint_{\widehat{F}} \theta_{F} \mu \ \ud \eta \ud \si\right)^{p}}.
$$
However, one should note that this is not directly a minimization of the sum in the left hand side of \eqref{embedding with theta}, because the collection $\F$ depends on the choice of the functions $\theta_Q$.

H\"older's inequality implies
$$
\frac{\| \theta_{F} \|_{L^{p}(\si ; \ \! l^{2})}^{p}}{\left(\iint_{\widehat{F}} \theta_{F} \mu \ \ud \eta \ud \si\right)^{p}} \geq \frac{1}{\|1_{\w F} \mu \|_{L^{p'}(\si ; \ \! \ell^2)}^p},
$$
and equality is reached with
$$
\theta_F= |1_{\w F} \mu |_{\ell^2}^{p'-2} 1_{\w F}\mu.
$$
This is the definition given in \eqref{def. of testing functions}.

One may wonder what happens if one tries to find in this way a family $\{\vartheta_Q\}_{Q \in \D}$ of test functions in place of the indicators in \eqref{theta testing 2}. Then one would be led to minimize the ratios
$$
\frac{\| \vartheta_G \|_{L^{p'}(\omega)}^{p'}}{\left( \int_G\vartheta_G\ud \omega \right)^{p'}}
$$
for $G$ in some collection of dyadic cubes. Again by H\"older's inequality this is minimized by $\vartheta_G=1_G$.

\section{Open problems and discussion}\label{discussion}

As mentioned in Introduction, the problem of this paper arose when we tried to build  two-weight $L^p$-theory for the Hilbert transform. One part in the existing $L^2$-theory is to bound the so-called \emph{tail form}. In \cite{Hytonen:Hilbert}, Section 6, this part is reduced to an estimate of the form
$$
(f,g) \mapsto \sum_{Q \in \D} \lambda_Q \int_{Q_+} f \ud  \sigma \int_{Q_-} g \ud \omega \lesssim \| f \|_{L^2(\si)} \| g \|_{L^{2}(\omega)},
$$
where $Q_+$ and $Q_-$ are two distinguished child cubes of the cube $Q \in \D$. The estimate
\begin{equation}\label{question2}
\sum_{Q \in \D} \lambda_{Q} \iint_{\widehat{Q}} f \mu \ \ud \eta \ud \si \int _{Q} g \ud \omega \lesssim \| f \|_\nof \| g \|_\nog
\end{equation}
came up as a model problem when we considered possible $L^p$-generalizations related to the tail form.

We raise here the following question: If $p \in (1,2)$, when does estimate \eqref{question2} hold? In order to suggest one approach, we briefly describe some results related to the operator
$S^\si f := \sum_{Q \in \D} \lambda_Q \int_Q f \ud \si 1_Q.$
Let $S^\omega$ be the corresponding operator with the measure $\omega$.

The $L^2$-result from \cite{NTV:bilinear} and its generalization in \cite{LSU:positive}, mentioned in Introduction, state that if $1<p \leq q < \infty$, then a similar theorem as Theorem \ref{thm:main} characterizes boundedness of  $S^\si$ from $L^p(\si)$ into $L^q(\omega)$; $S^\si$ is bounded if and only if $S^\si$ and $S^\omega$ satisfy a testing condition with indicators $1_Q$ of dyadic cubes  $Q \in \D$. Boundedness of $S^\si \colon L^p(\si) \to L^q(\omega)$ in the range $1<q<p<\infty$ was characterized by Tanaka \cite{Tanaka:upper} in terms of a discrete Wolff's potential. See also a unified theorem for all exponents $p,q \in (1,\infty)$ by H\"anninen, Hyt\"onen and Li \cite {HHL}. In \cite{HHL} it was shown that the indicator testing conditions do not imply boundedness of $S^\si \colon L^p(\si) \to L^q(\omega)$ if $1<q<p<\infty$; this example is even in the case when both the measures $\sigma$ and $\omega$ are equal to the Lebesgue measure.

The estimate \eqref{question2} can be equivalently formulated as the boundedness of 
$$
T^\sigma f:= \sum_{Q \in \D}  \lambda_{Q} \iint_{\widehat{Q}} f \mu \ \ud \eta \ud \si1_Q
$$
from $L^p(\si; \ \! \ell^2)$ into $L^p(\omega)$. The fact that we were not able to characterize the estimate \eqref{question2} when $p \in (1,2)$ somewhat fits to the known results about  boundedness of $S^\sigma \colon L^p(\si) \to L^q (\omega)$ and the relative order of the exponents $p$ and $q$. Namely, with the operator $S^\sigma$ the indicator testing conditions imply boundedness only when $1<p \leq q<\infty$, and in the range $p \geq 2$ where we can characterize boundedness of $T^\si$,  the exponent $p$ related to $L^p(\omega)$  is greater than or equal to both the exponents $2$ and $p$ related to $L^p(\si ; \ \! \ell^2)$. 

The approach to studying boundedness of $T^\si \colon L^p(\si ; \ \! \ell^2) \to L^p(\omega), p \in (1,2)$, that we propose here, is to use a certain \emph{quadratic testing} that was introduced by the second author in \cite{Vuorinen:well_loc, Vuorinen:dyadic_shift_Lp}, and was introduced to the second author by the first author in connection with a PhD project. It was shown in \cite{Vuorinen:well_loc} that if $p,q \in (1, \infty)$, then $S^\si \colon L^p(\si) \to L^q(\omega)$ is bounded if and only if  $S^\si$ and $S^\omega$ satisfy the quadratic testing condition.


Suppose $p \in (1,\infty)$. Define again the test functions
\begin{equation}\label{def. of testing functions 2}
\varphi_{Q}:= |1_{\widehat{Q}} \mu |_{\ell^{2}}^{p'-2}1_{\widehat{Q}} \mu, \quad Q \in \D.
 \end{equation}
For $Q \in \D$ let $T^\si_Q$ to be the corresponding localized version of the operator,  defined with the sum extending only over $Q' \subset Q$.  Let $\mathscr{T}^\si_p$  be the smallest possible constant such that
\begin{equation}\label{quadratic1}
\Big \| \Big( \sum_{Q \in \D} \big (a_Q T^\si_Q \varphi_Q \big)^2 \Big)^\frac{1}{2} \Big \|_{L^p(\omega)}
\leq \mathscr{T}^\si_p \Big \| \Big( \sum_{Q \in \D}  (a_Q \varphi_Q )^2 \Big)^\frac{1}{2} \Big \|_\nof
\end{equation}
holds for all collections $\{a_Q\}_{Q \in \D}$ of real numbers, with the understanding that $\mathscr{T}^\si_p$  may be $\infty$. We say that $T^\si$ satisfies the quadratic testing condition in $L^p$  if the constant $\mathscr{T}^\si_p$ is finite. Similarly, we define the quadratic testing constant $\mathscr{T}^\omega_{p'}$ for the formal adjoint operator $T^\omega g:= \sum_{Q \in \D} \lambda_Q \int_Q g \ud \omega 1_{\w Q} \mu$ using indicators $1_Q, Q \in \D,$ as test functions.

The precise question we want to ask is the following:

\begin{q}
Let $p \in (1,2)$. If $\mathscr{T}^\si_p +\mathscr{T}^\omega_{p'}< \infty$, then does it follow that $T^\si \colon L^p(\si ; \ \! \ell^2) \to L^p(\omega)$ is bounded, and that the estimate
$$
\| T^\si f \|_{L^p(\omega)} \lesssim (\mathscr{T}^\sigma_p+ \mathscr{T}^{\omega}_{p'})\| f \|_\nof, \quad f \in L^p(\si ; \ \! \ell^2),
$$
holds?
\end{q}
 
We remark that when $p \in (1,2)$ it is possible that one should use some other class of test functions than the ones defined in \eqref{def. of testing functions 2}; when using quadratic testing the proof does not offer a similar situation as described in the end of Section \ref{proof of main thm} to guess the test functions.

It is not immediately obvious that  the quadratic testing condition is a necessary consequence of boundedness of $T^\si$; nevertheless it follows in the spirit of a classical theorem by Marcinkiewicz and Zygmund \cite{MaZy} that if $T^\si \colon L^p(\si ; \ \! \ell^2) \to L^p(\omega)$ is bounded, then $\mathscr{T}^\sigma_p+\mathscr{T}^\omega_{p'} \lesssim \| T^\sigma\|_{L^p(\sigma ; \ \! \ell^2) \to L^p(\omega)}$, see \cite{Vuorinen:well_loc} or \cite{Vuorinen:dyadic_shift_Lp}.

\subsection*{Two-weight inequality of the Hilbert transform}

Finally, we state our conjecture about the two-weight inequality of the Hilbert transform in $L^p$. In the following we assume that $\sigma$ and $\omega$ are non-negative locally finite Borel measures in $\R$. We shall somewhat imprecisely just talk about the Hilbert transform as an operator $H^\si$ or $H^\omega$, where $\sigma$ and $\omega$ refer to the measure of integration in the definition of these operators. The operators $H^\si$ and $H^\omega$ should be thought of as formal adjoints of each other, in the sense that
$$
\int_\R g H^\si (f) \ud \omega = -\int_\R f H^\omega(g)  \ud \si
=\iint_{\R\times\R} \frac{g(x)f(y)}{x-y}\ud \omega(x)\ud \si (y)
$$
for $f$ and $g$ in a suitable class of functions. We refer to \cite{Lacey:2wHilbert,LSSU:2wHilbert} and \cite{Hytonen:Hilbert} for a precise definition of the Hilbert transform in this two-weight setting.  

Let $p \in (1, \infty)$. We say that $H^\sigma$ satisfies the global quadratic testing condition in $L^p$ if there exists a constant $C$ such that the inequality 
\begin{equation}\label{quadratic for H}
\Big \| \Big( \sum_{i =1}^\infty \big (a_i H^\si 1_{I_i} \big)^2 \Big)^\frac{1}{2} \Big \|_{L^p(\omega)}
\leq C \Big \| \Big( \sum_{i=1}^\infty  (a_i 1_{I_i} )^2 \Big)^\frac{1}{2} \Big \|_{L^p(\si)}
\end{equation}
holds for all collections $\{I_i\}_{i=1}^\infty$ of intervals in $\R$ and all collections $\{a_i \}_{i=1}^\infty$ of real numbers. The smallest possible constant $\mathscr{H}^\si_p$ in this inequality is the quadratic testing constant for $H^\si$. Similarly, we define the testing constant $\mathscr{H}^\omega_{p'}$ for the operator $H^\omega$ by replacing $H^\si$ with $H^\omega$, $p$ with $p'$ and reversing the roles of $\sigma$ and $\omega$ in \eqref{quadratic for H}.

In \cite{Lacey:2wHilbert,LSSU:2wHilbert} and \cite{Hytonen:Hilbert} it was shown that the two-weight inequality 
\begin{equation}\label{two weight in L^2}
\| H^\si f \|_{L^2(\omega)} \leq C \| f \|_{L^2(\sigma)}, \quad f \in L^2(\si),
\end{equation}
holds if and only if $H^\si$ and $H^\omega$ satisfy a global indicator testing condition if and only if $H^\si$ and $H^\omega$ satisfy a  local indicator testing condition and Muckenhoupt-Poisson two-weight $A_2$ condition holds. Moreover, the smallest constant $\mathscr{N}_2$ in \eqref{two weight in L^2} satisfies
$$
\mathscr{N}_2 \simeq \mathscr{H}^\si_2+ \mathscr{H}^\omega_{2}, 
$$
and there is also an equivalence with suitable local testing constants and a  two-weight $A_2$ constant; see the cited papers for details.
We remark that in $L^2$ the quadratic testing conditions are equivalent with the indicator testing conditions; hence we use the same notation for the constants.

\begin{con}\label{conj. for H}
Let $p \in (1,\infty)$. There exists a constant $C$ such that the Hilbert transform $H^\si$ satisfies the two-weight inequality
\begin{equation}\label{two_weight_Hilbert}
\| H^\si f \|_{L^p(\omega)} \leq C \| f \|_{L^p(\sigma)}, \quad f \in L^p(\si),
\end{equation}
if and only if $H^\si$ and $H^\omega$ satisfy the global quadratic testing conditions in $L^p$ and $L^{p'}$, respectively.  
Moreover, the smallest possible constant $\mathscr{N}_p$ in \eqref{two_weight_Hilbert} satisfies
$$
\mathscr{N}_p \simeq_p \mathscr{H}^\si_p+ \mathscr{H}^\omega_{p'}. 
$$
\end{con}

We also think that the two-weight inequality can be characterised in terms of certain local testing conditions together with a condition on the measures. 
Let $p \in (1, \infty)$. We say that $H^\sigma$ satisfies the \emph{local} quadratic testing condition if for some constant $C$
\begin{equation}\label{local quadratic for H}
\Big \| \Big( \sum_{i =1}^\infty \big (a_i 1_{I_i}H^\si 1_{I_i} \big)^2 \Big)^\frac{1}{2} \Big \|_{L^p(\omega)}
\leq C \Big \| \Big( \sum_{i=1}^\infty  (a_i 1_{I_i} )^2 \Big)^\frac{1}{2} \Big \|_{L^p(\si)}
\end{equation}
holds for all collections $\{I_i\}_{i=1}^\infty$ of intervals in $\R$ and all collections $\{a_i \}_{i=1}^\infty$ of real numbers.
The dual condition for $H^\omega$ is obtained from \eqref{local quadratic for H} by replacing $p$ by $p'$ and reversing the roles of the measures. We denote the smallest
possible constants by $\mathscr{H}^\si_{p, loc}$ and $\mathscr{H}^\omega_{p',loc}$.

If there exists a constant $C$ such that
\begin{equation}\label{eq:QMP1}
\Big \| \Big( \sum_{i=1}^\infty  \big( 1_{I_i} \int_{I_i^c} \frac{|f_i(x)| }{|x-c_i|} \ud \si (x) \big)^2 \Big)^{1/2} \Big \|_{L^p(\omega)} 
\le C \Big \| \Big( \sum_{i=1}^\infty |f_i|^2 \Big)^{1/2} \Big \|_{L^p(\si)}
\end{equation}
holds for all collections $\{I_i\}_{i=1}^\infty$ of intervals and all collections $\{f_i\}_{i=1}^\infty$ of functions, 
we say that the pair $(\si, \omega)$ of measures satisfies a (one-sided) quadratic Muckenhoupt-Poisson two-weight $A_p$ condition.
Here $I^c$ is the complement of the interval $I$.
We denote the smallest constant $C$ in \eqref{eq:QMP1} by $[\si,\omega]_p$. We will also need the dual condition 
of \eqref{eq:QMP1}, which is  obtained from \eqref{eq:QMP1} by replacing $p$ by $p'$ and reversing the roles of the measures. 
The best constant in the dual inequality is denoted by $[\omega, \si]_{p'}$.
We note that these conditions are positive in the sense that there is no cancellation involved.

We still formulate one condition, which we define to be satisfied if and only if there exists a constant $C$ such that
\begin{equation}\label{eq:AdjTest}
\begin{split}
\sum_{i=1}^\infty \big|& \langle a_i H^\si 1_{I_i}, b_i 1_{J(I_i)} \rangle_\omega \big| \\
&\le C \Big \| \Big( \sum_{i=1}^\infty |a_i1_{I_i}|^2 \Big)^{1/2} \Big \|_{L^p(\si)}
\Big \| \Big( \sum_{i=1}^\infty |b_i1_{J(I_i)}|^2 \Big)^{1/2} \Big \|_{L^{p'}(\omega)}
\end{split}
\end{equation}
holds for all collections $\{I_i\}_{i=1}^\infty$ of intervals and all collections $\{a_i\}_{i=1}^\infty$, $\{b_i\}_{i=1}^\infty$ of real numbers. 
Here $J(I_i)$ is an interval of equal length with $I_i$ and  adjacent to $I_i$. If \eqref{eq:AdjTest} is satisfied, 
let us denote the smallest constant $C$ in the inequality by $\scrT$.

In addition to Conjecture \ref{conj. for H} we suspect that the two-weight inequality \eqref{two_weight_Hilbert} holds if and only if the local quadratic testing conditions, 
the quadratic Muckenhoupt-Poisson two-weight $A_p$ condition and the condition \eqref{eq:AdjTest} hold, 
and that we have the estimate
\begin{equation}\label{eq:LocConj}
\mathscr{N}_p 
\simeq_p \mathscr{H}^\si_{p, loc} + \mathscr{H}^\omega_{p',loc} + [\si, \omega]_p+[\omega,\si]_{p'} +\scrT.
\end{equation}
At the moment it is unclear whether \eqref{eq:AdjTest} follows from \eqref{local quadratic for H}, \eqref{eq:QMP1} and their dual statements.
In any case, all the conditions are necessary for the two-weight inequality, and the estimate ``$\gtrsim_p$'' in \eqref{eq:LocConj} holds. 
The known $L^2$-result is analogous to \eqref{eq:LocConj}, except that in the $L^2$-case the condition \eqref{eq:AdjTest} is known to follow from the Muckenhoupt-Poisson $A_2$ condition.

\bibliography{weighted}
\bibliographystyle{abbrv}

\end{document}